\documentclass[final,3p,times]{elsarticle}

\usepackage{amssymb}
\usepackage{amsthm}
\usepackage{amsmath}
\usepackage{amsfonts}

\newtheorem{thm}{Theorem}[section]

\newtheorem{pr}[thm]{Proposition}

\numberwithin{equation}{section}

\begin{document}

\begin{frontmatter}



\title{Limit theorems on counting measures for a branching random walk  with immigration in a random environment}

\author[label1]{Mengxue Li}
\ead{}
\author[label2]{Chunmao Huang}
\ead{cmhuang@hitwh.edu.cn}

\author[label1]{Xiaoqiang Wang}
\ead{xiaoqiang.wang@sdu.edu.cn}
\address[label1]{School of Mathematics and Statistics, Shandong University, Weihai, Shandong, 264209, China}

 \address[label2]{Department of Mathematics, Harbin Institute of Technology (Weihai), Shandong, 264209, China}


\begin{abstract}
We consider a branching random walk with immigration in a  random environment, where the  environment is a stationary and ergodic sequence indexed by time. We focus on  the asymptotic properties of the sequence of measures $(Z_n)$ that count the number of particles of generation $n$ located in a Borel set of real line. In the present work, a series of limit theorems related to the above counting measures are established, including a central limit theorem, a moderate deviation principle and a large deviation result as well as a convergence theorem of the free energy.

\end{abstract}




\end{frontmatter}


\section{Introduction and main results}
As one of the frontier field of stochastic processes, the branching random walk with a random environment in time (BRWRE) has been making a lot of progress in recent years. For instance, Gao and Liu (2018)\cite {gl18} generalized the asymptotic expansions in the central theorem for BRWRE and obtained related results in the second and third orders. Wang and Huang (2019)\cite {wlll19} gave the sufficient and necessary conditions for existence of quenched moments as well as weighted moments of BRWRE. Huang {\it et al.}(2019) \cite{huang20} established large and moderate deviations of BRWRE in high-dimensional real space. 
As an extension of BRWRE, a branching random walk with immigration in  random environment in time (BRWIRE) also has attracted extensive attention. However there are too few relevant  results constraining some applications, hence it is interesting to study BRWIRE both in theory and in applications. In this paper we consider a branching random walk with immigration in a time-dependent random environment. Let $\xi=(\xi_n)$ representing the random environment in time be a stationary and ergodic sequence of random variables of distribution $\tau$. Suppose that each realization of $\xi_n$ is related to two distributions $\eta_n=\eta(\xi_n)$ and $\iota_n=\iota (\xi_n)$ on $\mathbb{N}\times\mathbb{R}.$

Given the environment $\xi$, the process can be described as follows. At time 0, there exists an initial particle $\emptyset$ of generation 0, located at $S_\emptyset=0\in \mathbb{R}.$ At time 1, it is replaced by new born particles with the relative displacements $L_i=L_i(\emptyset), i=1,2,\cdots,N(\emptyset),$ where $N(\emptyset)$ is the number of offspring of $\emptyset$. At the same time, $V_0$ new immigrants come and join located at $S_{0_0i},i=1,2,\cdots,V_0.$ All of the new born particles and new immigrants make up the first generation of particles. In general, each particle $u$ of generation $n$ with locations $S_u$ has $N(u)$ offspring $ui, i=1,2,\cdots,N(u),$ located at $S_{ui}=S_u+L_i(u).$ At the same time, $V_n$ new immigrants $0_ni, i=1,2,\cdots,V_n$ come and join with locations $S_{0_ni}$. Then all the offspring and new immigrants form the particles of generation $n+1$. In addition, we define that the random vector $X(u)=(N(u),L_1(u),L_2(u),\cdots)$ is of distribution $\eta_n$ and that the random vector $Y_n=(V_n,S_{0_n1},S_{0_n2},\cdots)$ is of distribution $\iota_n$. And they are all independent of each other conditional on $\xi.$
Let $\mathbb{P}_\xi$ so-called quenched law be the conditional probability given the environment $\xi$. The total probability can be written as $\mathbb{P}(\mathrm dx,\mathrm d\xi)=\mathbb{P}_\xi(\mathrm dx) \tau(\mathrm d\xi)$ and it is usually called the annealed law. Also, let $\mathbb{P}_{\xi,Y}$ be the conditional probability given the environment $\xi$ and the immigrant sequence $Y=(Y_n)$.  $\mathbb E_{\xi,Y}$, $\mathbb E_\xi $ and $\mathbb E$ denote the expectations with respect to $\mathbb{P}_{\xi,Y}$, $\mathbb{P}_\xi$ and $\mathbb P$ respectively.

Let $\mathbb{T}$ be the family tree with defining elements $\{N(u)\}$ and $\mathbb{T}_n=\{u\in\mathbb{T}:|u|=n\}$ be the set of particles of generation $n$ with $|u|$ representing the length of $u$. For $n \in \mathbb{N},$ define 
$$Z_n(\cdot)=\sum_{u\in\mathbb{T}_n}\delta_{S_u}(\cdot),$$
the counting measure of particles of generation $n$. In this paper, we concentrate on the asymptotic properties of the sequence of counting measures $(Z_n)$ via establishing limit theorems. Similarly, let $\mathbb{T}^X$ be the family tree without immigration and $\mathbb{T}_n^X=\{u\in\mathbb{T}^X:|u|=n\}$ be the set of particles of generation $n$ originated from the initial particle $\emptyset$. For convenience, we use the symbol  \lq{}$-$\rq{} for notations without immigration. According to the Laplace transform of $Z_n(\cdot)$, for $n \in \mathbb{N}$ and $t\in\mathbb{R},$ we define 
\begin{equation}
\tilde Z_n(t)=\sum_{u\in\mathbb T_n}e^{tS_u},\quad
W_n(t)=\frac{\tilde Z_n(t)}{\Pi_n(t)},\quad
\bar Z_n(t)=\sum_{u\in\mathbb T^X_n}e^{tS_u}\quad\text{and}\quad \bar W_n(t)=\frac{\bar Z_n(t)}{\Pi_n(t)}
\end{equation}
\begin{equation}
m_n(t)=\mathbb E_\xi\sum_{i=1}^{N(u)}e^{tL_i(u)}\;\; (u\in\mathbb {N^*}^n),\quad Y_n(t)=\sum_{i=1}^{V_n}e^{tS_{0_ni}},\quad
\Pi_0(t)=1\quad\text{and}\quad \Pi_n(t)=\prod_{i=0}^{n-1}m_i(t)\quad (n\geq 1),
\end{equation}
In particular, set $m_n(0)=m_n$, $\Pi_n=\Pi_n(0)$, $W_n=W_n(0)$ and $\bar W_n=\bar W_n(0)$  for short. We can see that $\bar W_n(t)$ is the intrinsic non-negative martingale of BRWRE, hence it converges almost surely  (a.s.) to some limit $\bar W(t)$ with $\mathbb E_\xi\bar W(t)\leq 1$. Also, it is not difficult to verify that $W_n(t)$ is a non-negative sub-martingale with respect to the probability $\mathbb P_{\xi, Y}$. We consider that $(Z_n)$ is a supercritical branching process i.e. $\mathbb E\log m_0>0$. Based on the necessary moment conditions (for example, $\mathbb E\log^+ V_0<\infty$ and $\mathbb E\log^+\mathbb E_\xi N^p<\infty$ for some $p>1$), the sub-martingale $W_n$ converges a.s. to a limit $W\in(0,\infty)$. To study the intrinsic properties and for simplicity, we further assume that 
\begin{equation}\label{3}
\mathbb P_\xi(N=0)=0,\qquad  \mathbb P_\xi(N=1)<1\qquad  \text{and}\qquad \mathbb E_\xi \sum_{i=1}^{N}L_i=0\qquad a.s..
\end{equation}

Firstly, we establish a central limit theorem concerning $(Z_n)$. Set 
$$\sigma_n^2=\frac{1}{m_n}\mathbb E_\xi \sum_{i=1}^{N(u)}L_i(u)^2 \quad (u\in\mathbb {N^*}^n) \quad \text{and} \quad b_n=\left(\sum_{i=0}^{n-1}\sigma_i^2\right)^{1/2}$$.

\begin{thm}[Central limit theorem]\label{tclt}
 If $\mathbb{E}\sigma_0^2 \in(0,\infty)$, $\mathbb E\log^+{Y_0(\pm\delta)}<\infty$ for some $\delta>0$ and  $\mathbb E\log^+\mathbb E_\xi  N^p<\infty$ for some $p>1$,
then  a.s., 
 \begin{equation*}\label{LTET10.2}
\frac{ Z_n(-\infty,b_nx]}{ Z_n(\mathbb{R})}\rightarrow  \Phi(x)\qquad  \forall x\in\mathbb{R},
\end{equation*}
  where $\Phi(x)=\frac{1}{\sqrt{2\pi}}\int_{-\infty}^x e^{-t^2/2} \mathrm dt$.
\end{thm}

We next establish a moderate deviation principle associated to $(Z_n)$.   Let $a_n=n^\alpha$ with $\alpha\in(0.5,1)$.

\begin{thm}[Moderate deviation principle]\label{MDP} 
 If $\mathbb E \frac{1}{m_0}\sum_{i=1}^N e^{\delta|L_i|}<\infty$, $\mathbb E\log^+{Y_0(\pm\delta)}<\infty$  and $\mathbb E\log^+\mathbb E_\xi(\sum_{i=1}^Ne^{\delta|L_i|})^p<\infty$ for some  $\delta>0$ and $p>1$,
then it is a.s.   that the sequence of probabilities
$A \mapsto  {Z_n(a_nA)}/Z_n(\mathbb R)$ satisfies a principle of moderate deviation with
rate function $\frac{x^2}{2\sigma^2}$: for each measurable subset $A$ of
$\mathbb{R}$, 
\begin{eqnarray*}\label{LDP1}
   - \frac{1}{2\sigma^2}\inf_{x\in A^\circ} x^2
   \leq \liminf_{n\rightarrow \infty}
             \frac{n}{a_n^2}  \log \frac{Z_n (a_nA)}{Z_n(\mathbb R)} 
   \leq  \limsup_{n\rightarrow \infty}
              \frac{n}{a_n^2}  \log \frac{Z_n (a_nA)}{Z_n(\mathbb R)}
   \leq - \frac{1}{2\sigma^2}\inf_{x\in \bar A} x^2,
\end{eqnarray*}
where  $\sigma^2=\mathbb E\sigma_0^2$,  $A^\circ$ denotes the interior of $A$, and $\bar A$ its closure.
\end{thm} 

Finally, we consider large deviations of $\log Z_n$. We shall begin with the study of convergence of the free energy $\frac{\log \tilde{Z}_n}{n}$.
Assume that
\begin{equation} \label{H1}
 E|L_1| < \infty, \quad
\mathbb{E}|\log m_0 (t)|<\infty,\quad  \mathbb{E}\left|\frac{m_0' (t)}{m_0 (t)} \right|<\infty\quad\text{and}\quad\mathbb E|\log Y_0(t)|<\infty
\end{equation}
for all $t\in \mathbb{R}$. Therefore, the function $\Lambda (t):=\mathbb E\log m_0(t)$
is well-defined and differentiable everywhere on $\mathbb{R}$.
Define
\begin{equation}\label{LS1.2t}
 t_-= \inf \{ t \in \mathbb{R}:t\Lambda'(t)-\Lambda(t) \leq 0\} \quad \text{and}\quad
t_+= \sup \{ t \in \mathbb{R}:  t\Lambda'(t)-\Lambda(t) \leq 0\}.
\end{equation}
For BRWRE, it is known from \cite[Theorem 3.1]{hl14} that a.s. for all $t\in\mathbb{R}$,
\begin{equation}\label{LTE3.3}
\lim_{n\rightarrow \infty} \frac{1}{n}\log \bar Z_n(t) =
\bar\Lambda (t) :=
   \left\{
   \begin{array} {lll}
          \Lambda (t)  & if &   t\in (t_-, t_+) ,\\
            t\Lambda'(t_+) & if & t\geq t_+, \\
            t\Lambda'(t_-) & if & t\leq t_-.
   \end{array}
     \right.
\end{equation}
\begin{thm}[Convergence of the free energy]\label{tfn}
Assume  \eqref{H1}. Let $\tilde\Lambda (t) :=\max\{\bar \Lambda(t),0\}$.
It is a.s.   that for all $t\in\mathbb{R}$,
\begin{equation}\label{fn}
\lim_{n\rightarrow \infty} \frac{1}{n}\log \tilde Z_n(t) =
\tilde\Lambda (t).
\end{equation}
\end{thm}

Applying  Theorem \ref{tfn} and the  G\"artner-Ellis theorem, we obtain the following large deviation results.
 Denote
$t_1=\inf\{t\in\mathbb R:\Lambda(t)<0\}$ and $t_2=\sup\{t\in\mathbb R:\Lambda(t)<0\}$.
According to the relationship between $t_1,t_2, t_-$ and $t_+$ and considering the properties of the function $\Lambda(t)$, we can distinguish three cases: 
case I. $t_1=-\infty$ and $t_2=\infty$; case II. $t_1<t_-<t_2<0$; case III. $0<t_1<t_+<t_2$.

\begin{thm}[Large deviations]\label{tld}
Assume \eqref{H1}. Let $\tilde\Lambda^*(x)= \sup_{t\in \mathbb{R}} \{ xt -  \tilde\Lambda (t)\}$ be the Legendre transform of
$\tilde\Lambda(t)$.
\begin{itemize}
\item[(a)]For case I, it is a.s.   that the sequence of measures
$A \mapsto  Z_n(nA)$ satisfies a principle of  large deviation with
rate function $\tilde \Lambda^* (x)$: for each measurable subset $A$ of
$\mathbb{R}$, 
\begin{eqnarray*}\label{LDP1}
   - \inf_{x\in A^\circ} \tilde\Lambda^* (x)
   \leq \liminf_{n\rightarrow \infty}
             \frac{1}{n}  \log Z_n (nA)
  \leq   \limsup_{n\rightarrow \infty}
            \frac{1}{n} \log   Z_n (nA)
   \leq - \inf_{x\in \bar A}  \tilde\Lambda^*(x),
\end{eqnarray*}
where  $A^\circ$ denotes the interior of $A$, and $\bar A$ its closure.
\end{itemize}

\item[(b)]For cases II and III, it is a.s. that for each closed subset $F$ of
$\mathbb{R}$,
$$ \limsup_{n\rightarrow \infty}
            \frac{1}{n} \log   Z_n (nF)
   \leq - \inf_{x\in F}  \tilde \Lambda^*(x),$$
and for each open subset $G$ of
$\mathbb{R}$,
$$ \liminf_{n\rightarrow \infty}\frac{1}{n}  \log Z_n (nG)\geq - \inf_{x\in G\cap \mathcal E} \tilde\Lambda^* (x),$$
where $\mathcal E=(\Lambda'(t_2) , \Lambda' (t_+))$ for case I and  $\mathcal E=(\Lambda'(t_-) , \Lambda' (t_1))$ for case III, is the set of exposed points of the rate function $\tilde \Lambda^*$.
\end{thm}

\section{Sketch of proof} 
The proofs  will be based on the convergence of the sub-martingale $W_n(t)$. The following proposition describes the a.s. and $L^p$-convergence of $W_n(t)$ under the probability $\mathbb P_{\xi, Y}$ as well as its $L^p$-convergence rate.
For $t\in\mathbb R$ fixed, set 
$f_t(x)=\frac{1}{x}\Lambda(xt)-\Lambda(t)$.

\begin{pr}\label{lp}
Let $p>1$. 
If $\Lambda(t)>0$, $\mathbb E\log^+\frac{Y_0(t)}{m_0(t)}<\infty$, $f_t(p)<0$ and $\mathbb E\log^+ \mathbb E_\xi \bar Z_1(t)^p<\infty$, then
$\sup_n\mathbb E_{\xi, Y}W_n(t)^p<\infty$ a.s.,
so that $W_n(t)$ converges a.s. and in $L^p$ under $\mathbb P_{\xi, Y}$ to some limit $W(t)$ which satisfies the decomposition
\begin{equation}\label{dl}
W(t)=\bar W(t)+\sum_{k=1}^\infty \Pi_k(t)^{-1}\sum_{i=1}^{V_{k-1}}\bar W(0_{k-1}i,t)e^{tS_{0_{k-1}i}},
\end{equation} 
where the notation $\bar W(u,t)$ denotes the limit of the intrinsic non-negative martingale $\bar W_n(u,t)$ of the BRWRE originated from the particle $u\in\mathbb U$, and the $L^p$-convergence rate is a.s.,
\begin{equation}\label{ecr}
\limsup_{n\rightarrow\infty}\frac{1}{n}\log \left(\mathbb E_{\xi, Y}|W_{n+1}(t)-W_n(t)|^p\right)^{1/p}\leq
\left\{\begin{array}{ll}
\max\{-\Lambda(t),f_t(p)\}&\text{if $p\in(1,2]$,}\\
\max\{-\Lambda(t),f_t(p), f_t(2)\}&\text{if $p> 2$.}
\end{array}
\right.
\end{equation}
\end{pr}

\begin{proof} We first prove  \eqref{ecr}, which implies $\sup_n{\mathbb{E}_{\xi,Y} W_n(t)<\infty}$ a.s. Here we just provide the proof for $p\in(1,2]$; for $p>2$, \eqref{ecr} can be obtained by a technique of induction.ee
According to the structure of the family tree $\mathbb{T}$, we decompose 
\begin{equation} \label{dss}
W_n(t)=\bar W_n(t)+\sum_{k=1}^{n} \Pi_k(t)^{-1}\sum_{i=1}^{V_{k-1}}\bar W_{n-k}(0_{k-1}i,t)\mathrm{e}^{tS_{0_{k-1}i}}.
\end{equation}
Set $Y_{-1}=(1,0,0,\cdots)$ and $\bar W_n(0_{-1}i,t)=\bar W_n(t)$.
Using Minkowski's inequality and Burholder's inequality, we get
\begin{equation*} \label{CRE 2.2.26}
\left(\mathbb{E}_{\xi,Y}\left| W_{n+1}(t)-W_n(t) \right|^p \right)^{{1}/{p}} \leq C \left( \left(\mathbb{E}_{T^n\xi}\left|\bar{W}_1(t)-1\right|^p\right)^{{1}/{p}} \sum_{k=0}^{n}\Pi_k(t)^{-1} Y_{k-1}(t)\gamma_{k,n}(p,t) +\Pi_{n+1}(t)^{-1} Y_n(t)\right),
\end{equation*} 
where $\gamma_{k,n}(p,t)=\prod_{i=k}^{n-1}[m_i(t)^{-1}m_i(pt)^{1/p}]$, $C>0$ is a general constant and
the notation $T$ represents the shift operator: $T \xi$ =$\left(\xi_1,\xi_{2},\cdots \right)$ if $\xi=\left(\xi_0,\xi_1,\cdots\right)$. Applying \cite[Lemma 2.1]{wh19} leads to
\begin{equation*} \label{CRE 2.2.28}
\limsup_{n\rightarrow \infty} \frac{1}{n} \log \left[\left(\mathbb{E}_{T^n\xi}\left|\bar{W}_1(t)-1\right|^p\right)^{{1}/{p}}\sum_{k=0}^{n} \Pi_k(t)^{-1} Y_{k-1}(t) \gamma_{k,n}(p,t)\right]  \leq \max \left\{-\Lambda(t), f_t(p)\right\} \qquad a.s.
\end{equation*}
Besides, it is not difficult to see that 
$\limsup_{n\rightarrow \infty} \frac{1}{n} \log [\Pi_{n+1}(t)^{-1} Y_n(t)]\leq -\Lambda(t)$ a.s.
So \eqref{ecr} is proved. In order to obtain the decomposition
\eqref{dl}, we can prove that 
\begin{equation} \label{CRE 2.2.18}
\sum_{k=1}^{\infty}\Pi_k(t)^{-1} \sum_{i=1}^{V_{k-1}}\mathrm{e}^{tS_{0_{k-1}i}} \sup_n\bar W_n(0_{k-1}i,t)  < \infty \qquad a.s.
\end{equation}
Then let $n$ tend to infinity in  \eqref{dss} and use the dominated convergence theorem.
\end{proof}

\begin{proof}[Proof of Theorem \ref{tclt}]
By Proposition \ref{lp}, $W=lim_{n\rightarrow\infty}W_n$ exists and satisfies the decomposition $W=\bar W(0)+\sum_{k=1}^\infty \Pi_k^{-1}{V_{k-1}}\bar W(0_{k-1}i,0)$.
Let $\Psi_n(t)=\Pi_n^{-1}\tilde Z_n(\mathbf{i}t)$ be the characteristic function of the measure $\Pi_n^{-1} Z_n$.
It suffices to show that $\Psi_n(t/b_n)\rightarrow g(t) W$ a.s., where $g(t)$ is the characteristic function of the standard normal distribution. Using \eqref{dss} with $t$ replaced by $\mathbf{i}t/b_n$, we have
\begin{equation}\label{CTL delta0}
\Psi_n\left({t}/{b_n}\right) = \bar{\Psi}_n\left({t}/{b_n}\right) + \sum_{k=1}^{n} \Pi_k^{-1} \sum_{i=1}^{V_{k-1}} \bar{\Psi}_{n-k}\left(0_{k-1}i,{t}/{b_n}\right) \mathrm{e}^{\textbf{i}\frac{t}{b_n}S_{0_{k-1}i}},
\end{equation} 
where $\bar{\Psi}_{n-k}(0_{k-1}i,t)=\prod_{l=k}^{n-1}[m_l^{-1}m_l(\mathbf i t)]\bar W_{n-k}(0_{k-1}i,\mathbf{i}t)$. Since $\bar \Psi_n(t/b_n)\rightarrow g(t) \bar W(0)$  a.s. by  the central limit theorem for BRWRE (see \cite[Theorem 10.1]{hl14}),   we can prove that  $\bar{\Psi}_{n-k}\left(0_{k-1}i,{t}/{b_n}\right)\rightarrow g(t)\bar W(0_{k-1}i,0)$ a.s. Noticing \eqref{CRE 2.2.18}, we  let $n$ tend to infinity in  \eqref{CTL delta0} and use the dominated convergence theorem.
\end{proof}

\begin{proof}[Proof of Theorem \ref{MDP}]
Let $\Gamma_n(t) = \log \frac{\int \mathrm{e}^{tx} Z_n(a_n dx )}{Z_n(\mathbb{R})},$ then 
\begin{equation}\label{MDP1.4.1}
\frac{n}{a_n^2} \Gamma_n\left(\frac{a_n}{n}t\right) = \frac{n}{a_n^2} \log\frac{\Pi_n(\frac{a_n}{n}t)}{\Pi_n}+\frac{n}{a_n^2} \left[\log W_n\left(\frac{a_n}{n}t\right)-\log W_n\right].
\end{equation}	
By \cite[Lemma 6.1]{huang20}, $\lim_{n \rightarrow\infty} \frac{n}{a_n^2} \log [\Pi_n^{-1}\Pi_n(\frac{a_n}{n}t)] = \frac{1}{2} \sigma^2 t^2$ a.s.
Now we shall prove that $W_n(t)$ converges uniformly a.s. to the limit $W(t)$ on the interval $D_\varepsilon=[-\varepsilon,\varepsilon]$ for some $\varepsilon >0$. Notice that 
	$$\sup_{t\in D_\varepsilon} \left|W_{n+1}(t) - W_n(t)\right| \leq \left|W_{n+1} - W_n\right| + \int_{D_\varepsilon} \left|W'_{n+1}(t) - W'_n(t)\right|dt.$$ By  Proposition \ref{lp}, it can be seen that $\sum_n |W_{n+1} - W_n|<\infty$ a.s. It remains to show that \begin{equation} \label{UCI4.1.1}
	\sum_{n} \sup_{t\in D_\varepsilon}\left( \mathbb{E}_{\xi,Y} \left|W'_{n+1}(t) - W'_n(t)\right|^{p}\right)^{{1}/{p} }< \infty \qquad a.s.
	\end{equation}
for suitable $p \in (1,2]$. By the derivative of the decomposition \eqref{dss}  and using Minkowski's inequality, we obtain that  
	\begin{eqnarray*}  
	\sup_{t\in D_\varepsilon}\left( \mathbb{E}_{\xi,Y} \left|W'_{n+1}(t) - W'_n(t)\right|^{p}\right)^{{1}/{p}}
	 &\leq & \sum_{k=0}^{n} \sup_{t\in D_\varepsilon}\left\{\Pi_{k}(t)^{-1} Y_{k-1}(t) \left(\mathbb{E}_{T^k\xi} \left|\bar W'_{n+1-k}(t) - \bar W'_{n-k}(t)\right|^{p}\right)^{{1}/{p}}\right\}\\
	&& + \sum_{k=1}^{n} \sup_{t\in D_\varepsilon}\left\{\Pi_{k}(t)^{-1} \tilde{Y}_{k-1}(t) \left(\mathbb{E}_{T^k\xi} \left|\bar W_{n+1-k}(t) - \bar W_{n-k}(t)\right|^{p}\right)^{{1}/{p}}\right\}+ \sup_{t\in D_\varepsilon}\left\{\Pi_{n+1}(t)^{-1} \tilde{Y}_n(t)\right\}\\
	&=:& I_{1,n}+I_{2,n}+I_{3,n},
	\end{eqnarray*} 
where $\tilde Y_n(t)=\sum_{i=1}^{V_{n-1}} \mathrm{e}^{tS_{0_{n-1}i}} |S_{0_{n-1}i} - \sum_{l=0}^{n-1} m_l(t)^{-1}{m'_l(t)}|$. By tedious calculations, it can be seen that $\sum_n I_{i,n}<\infty$ a.s., $i=1,2,3$, then \eqref{UCI4.1.1} is proved.

\end{proof}

\begin{proof}[Proof of Theorem \ref{tfn}]
Since $\tilde Z_n(t)\geq\max\{\bar Z_n(t), Y_n(t)\}$, by \eqref{LTE3.3},  $$\liminf_{n\rightarrow\infty}\frac{1}{n}\log \tilde Z_n(t)\geq \max\{\liminf_{n\rightarrow\infty}\frac{1}{n}\log\bar Z_n(t), \liminf_{n\rightarrow\infty}\frac{1}{n}\log Y_n(t)\}=\max\{\bar \Lambda(t),0\}=\tilde \Lambda(t)\quad a.s.$$
It remains to show the superior limit. For $s>\max\{ \Lambda(t),0\}$,  we have  $\mathbb E_{\xi, Y}(\sum_n e^{-ns}\tilde Z_n(t))<\infty$ a.s., which implies that $\limsup_{n\rightarrow\infty}\frac{1}{n}\log \tilde Z_n(t)\leq \max\{\Lambda(t),0\}$ a.s. For $t\in(t_-,t_+)$, $\Lambda(t)=\bar \Lambda(t)$. So it remains to consider the case $t\geq t_+$ and $t\leq t_-$. In the following, we just provide the proof for the case $t\geq t_+$; the proof for $t\leq t_-$ is similar. 
Let $t\geq t_+$. Notice that $\limsup_{n\rightarrow\infty}\frac{1}{n}\log \mathbb E_{\xi, Y}\tilde Z_n(t)\leq \max\{\Lambda(t),0\}$. Denote $\Lambda_+^*(x)= \sup_{t\in \mathbb{R}} \{ xt -  \max\{\Lambda (t),0\}\}$  the Legendre transform of
$\max\{\Lambda (t),0\}$. By \cite[Theorem 4.5.3]{z}, for $a>\max\{\Lambda'(t_+),0\}$,
$\limsup_{n\rightarrow\infty}\frac{1}{n}\log \mathbb E_{\xi, Y}\tilde Z_n(t)\leq -\Lambda_+^*(a)<0$ a.s.,
which leads to $\sum_n \mathbb P_{\xi, Y}(Z_n[na,\infty)\geq1)<\infty$ a.s. Therefore, by Borel-Cantelli Lemma, $Z_n[na,\infty)=0$ a.s., so that $R_n<na$ a.s., where $R_n:=\max\{S_u: u\in \mathbb T_n\}$ is the rightmost position of generation $n$. For $t_0\in(t_+,t]$, noticing the fact that $\tilde Z_n(t)\leq \tilde Z_n(t_0)\exp\{(t-t_0)\frac{R_n}{n}\}$, we deduce that $\limsup_{n\rightarrow\infty}\frac{1}{n}\log \tilde Z_n(t)\leq \max\{\Lambda(t_0),0\}+(t-t_0)a$ a.s. The proof is finished by letting $a\downarrow \max\{\Lambda'(t_+),0\}$ and $t_0\uparrow t_+$.

\end{proof}





\end{document}